\documentclass[12pt]{article}
\title{A criterion for uniform finiteness in the imaginary sorts}
\author{Will Johnson}

\usepackage{amsmath, amssymb, amsthm}    	
\usepackage{fullpage} 	
\usepackage{amscd}
\usepackage{hyperref}
\usepackage{titlefoot}

\DeclareMathOperator*{\forkindep}{\raise0.2ex\hbox{\ooalign{\hidewidth$\vert$\hidewidth\cr\raise-0.9ex\hbox{$\smile$}}}}

\newtheorem{theorem}{Theorem}[section] 
\newtheorem{lemma}[theorem]{Lemma}

\newtheorem{proposition}[theorem]{Proposition}
\newtheorem*{theorem-star}{Theorem}
\newtheorem*{conjecture-star}{Conjecture}

\theoremstyle{definition}
\newtheorem{definition}[theorem]{Definition}
\newtheorem{example}[theorem]{Example}

\theoremstyle{remark}

\newtheorem{claim}[theorem]{Claim}

\newtheorem{observation}[theorem]{Observation}
\newtheorem*{acknowledgment}{Acknowledgments}

\newenvironment{claimproof}[1][\proofname]
               {
                 \proof[#1]
                 
               }
               {
                 \endproof
               }

\newcommand{\Qq}{\mathbb{Q}}
\newcommand{\eq}{\mathrm{eq}}

\newcommand{\Zz}{\mathbb{Z}}
\newcommand{\Nn}{\mathbb{N}}

\begin{document}
\maketitle\unmarkedfntext{

}

\begin{abstract}
  Let $T$ be a theory.  If $T$ eliminates $\exists^\infty$, it need
  not follow that $T^{\eq}$ eliminates $\exists^\infty$, as shown by
  the example of the $p$-adics.  We give a criterion to determine
  whether $T^{\eq}$ eliminates $\exists^\infty$.  Specifically, we show
  that $T^{\eq}$ eliminates $\exists^\infty$ if and only if
  $\exists^\infty$ is eliminated on all interpretable sets of ``unary
  imaginaries.''  This criterion can be applied in cases where a full
  description of $T^{\eq}$ is unknown.  As an application, we show that
  $T^{\eq}$ eliminates $\exists^\infty$ when $T$ is a C-minimal
  expansion of ACVF.
\end{abstract}

\section{Conventions}
\begin{definition}
  Let $X$ be a definable or interpretable set in an
  $\aleph_0$-saturated structure.  Say that \emph{$\exists^\infty$ is
    eliminated on $X$} if for every definable family $\{D_a\}_{a \in
    Y}$ of subsets of $X$, the following (equivalent) conditions hold:
  \begin{enumerate}
  \item \label{uno} The set $\{a \in Y : |D_a| = \infty\}$ is
    definable.
  \item \label{dos} There is an $n \in \Nn$ such that for all $a \in Y$,
    \[ |D_a| = \infty \iff |D_a| > n.\]
  \end{enumerate}
\end{definition}
In a non-saturated structure $M$, we use Condition \ref{dos}, which is
invariant under elementary extensions, and stronger than Condition
\ref{uno}.  In other words, we say that ``$\exists^\infty$ is
eliminated on $X$'' if this holds in an $\aleph_0$-saturated
elementary extension $M^* \succeq M$.  This is a slight abuse of
terminology.
\begin{definition}
  A theory $T$ has \emph{uniform finiteness} if $\exists^\infty$ is
  eliminated on every definable set.  We also say that \emph{$T$
    eliminates $\exists^\infty$.}
\end{definition}
In a 1-sorted theory, uniform finiteness is equivalent to elimination
of $\exists^\infty$ on the home sort, by the following observation.
\begin{observation}
  If $\exists^\infty$ is eliminated on $X$ and $Y$, it is eliminated
  on $X \times Y$.  In fact, $S \subseteq X \times Y$ is finite if and
  only if both of the projections $S \to X$ and $S \to Y$ have finite
  image.
\end{observation}
\begin{example}\label{e1}
  If $(M,\le,+)$ is a dense o-minimal structure, then $M$ eliminates
  $\exists^\infty$.  Indeed, a definable set $X \subseteq M$ is
  infinite if and only if $X$ has non-empty interior.
\end{example}
\begin{example}
  If $(K,+,\cdot)$ is a $p$-adically closed field, such as $\Qq_p$,
  then $K$ eliminates $\exists^\infty$.  In fact, a definable set $X
  \subseteq K$ is infinite if and only if it has interior, by work of
  Macintyre \cite{macintyre}.
\end{example}
\begin{example}\label{e3}
  The ordered abelian group $(\Zz,\le,+)$ does not eliminate
  $\exists^\infty$, because there is no uniform bound on the size of
  the finite intervals $[1,n]$.
\end{example}
\section{When does $T^{\eq}$ eliminate $\exists^\infty$?}
Uniform finiteness does not pass from $T$ to $T^{\eq}$.  In other
words, $\exists^\infty$ can be eliminated on definable sets without
being eliminated on interpretable sets.  This happens in $\Qq_p$,
which interprets $(\Zz,\le,+)$ as the value group.

In many theories, it is difficult to fully characterize interpretable
sets.  For example, in the theory of algebraically closed valued
fields (ACVF), the classification of interpretable sets is rather
complicated \cite{HHM}.  Moreover, this classification fails to
generalize to C-minimal expansions of ACVF \cite{unexpected}.

In Theorem~\ref{main-th}, we will give a relatively simple
criterion which can be used to show that $T^{\eq}$ eliminates
$\exists^\infty$ without first characterizing interpretable sets.  As
an application, we will show that $T^{\eq}$ eliminates
$\exists^\infty$ when $T$ is a C-minimal expansion of ACVF.

Assume henceforth that $T$ is one-sorted.
\begin{definition}
  In a model $M \models T$, a \emph{unary} definable set is a
  definable subset of $M = M^1$.
\end{definition}
\begin{definition}
  An interpretable set $X$ is a \emph{set of unary imaginaries} if
  there is a definable relation $R \subseteq X \times M$ such that the
  following map is an injection:
  \begin{equation*}
    x \mapsto R_x := \{m \in M : (x,m) \in R\}.
  \end{equation*}
\end{definition}
In other words, $X$ is a set of unary imaginaries if the elements of
$X$ are codes for unary definable sets, in some uniform way.
\begin{theorem}\label{main-th}
  Suppose that $\exists^\infty$ is eliminated on every set of unary
  imaginaries.  Then $T^{\eq}$ eliminates $\exists^\infty$.
\end{theorem}
\begin{proof}
  Let $M_0 \models T$ be a small model.  Let $N_0$ be the expansion of
  $M_0^{\eq}$ by a new sort $\Nn \cup \{\infty\}$ and functions
  \begin{align*}
    Y &\to \Nn \cup \{\infty\} \\
    a &\mapsto |D_a|
  \end{align*}
  for every definable family $\{D_a\}_{a \in Y}$ in $M_0^{\eq}$.  Let
  $N = (M^{\eq}, \Nn^* \cup \{\infty\})$ be an $\aleph_0$-saturated
  elementary extension of $N_0$.

  Then $\Nn^*$ is an $\aleph_0$-saturated elementary extension of
  $\Nn$, $M$ is an $\aleph_0$-saturated model of $T$, and every
  interpretable set $X$ in $M$ has a non-standard ``size''
  \[ |X| \in \Nn^* \cup \{\infty\}.\]
  Say that $X$ is \emph{pseudofinite} if $|X|$ is less than the symbol
  $\infty$.  (In particular, finite sets are pseudofinite.)  It
  suffices to show that every pseudofinite interpretable set is
  finite, because of the $\aleph_0$-saturation of $N$.

  Say that an interpretable set $X$ in
  $M$ is \emph{wild} if there is an infinite pseudofinite definable
  family of subsets of $X$.  Otherwise, say $X$ is \emph{tame}.  By
  assumption, $\exists^\infty$ is eliminated on sets of unary
  imaginaries.  Therefore, every pseudofinite set of unary imaginaries
  is finite.  Equivalently, $M^1$ is tame.
  \begin{claim}\label{unions}
    If $X$ is tame, so is any definable subset of $X$.  If $X$ and $Y$
    are tame, then so is $X \cup Y$.
  \end{claim}
  \begin{claimproof}
    The first statement is trivial.  For the second statement, let
    $\mathcal{D}$ be a pseudofinite definable family of subsets of $X
    \cup Y$.  Note that $\{D \cap X : D \in \mathcal{D}\}$ is
    \begin{itemize}
    \item \emph{pseudofinite}, because $\mathcal{D}$ is pseudofinite, and
    \item \emph{finite}, because $X$ is tame
    \end{itemize}
    Similarly, $\{D \cap Y : D \in \mathcal{D}\}$ is finite.  Finally,
    the map
    \begin{equation*}
      D \mapsto (D \cap X, D \cap Y)
    \end{equation*}
    yields an injection from $\mathcal{D}$ into a product of two
    finite sets.  Thus $\mathcal{D}$ is finite.
  \end{claimproof}

  \begin{claim}
    Let $\pi : X \to Y$ be a definable map with finite fibers.  If $Y$
    is tame, then so is $X$.
  \end{claim}
  \begin{claimproof}
    By saturation, there is a uniform upper bound $k$ on the size of
    the fibers.  We proceed by induction on $k$.  The base case $k =
    1$ is trivial.  Suppose $k > 1$.  Let $\mathcal{D}$ be a
    pseudofinite definable family of subsets of $X$.  Let
    \[ \mathcal{E} = \{\pi(D) : D \in \mathcal{D}\}\]
    and
    \[ \mathcal{F} = \{\pi(X \setminus D) : D \in \mathcal{D}\}\]
    Then $\mathcal{E}$ and $\mathcal{F}$ are both pseudofinite
    definable families of subsets of $Y$.  By tameness of $Y$, they
    are both finite.

    It remains to show that the fibers of $\mathcal{D} \to
    \mathcal{E} \times \mathcal{F}$ are finite.  Replacing
    $\mathcal{D}$ with such a fiber, we may assume that $\pi(D)$ and
    $\pi(X \setminus D)$ are independent of $D$, as $D$ ranges over
    $\mathcal{D}$.  Let $U = \pi(D)$ and $V = \pi(X \setminus D)$ for
    any/every $D \in \mathcal{D}$.  Let $Y' = U \cap V$ and $X' =
    \pi^{-1}(Y')$.  Then the map $D \mapsto D \cap X'$ is injective on
    $\mathcal{D}$, because every element $D$ of $\mathcal{D}$ contains
    $\pi^{-1}(U \setminus V)$ and is disjoint from $\pi^{-1}(V
    \setminus U)$.  So it suffices to show that $X'$ is tame.  Let $D$
    be some arbitrary element of $\mathcal{D}$.  Then $X' \cap D$ and
    $X' \setminus D$ each intersect every fiber of $X' \to Y'$, by
    choice of $X'$.  In particular, the two maps
    \[ X' \cap D \to Y'\]
    \[ X' \setminus D \to Y'\]
    have finite fibers of size less than $k$.  By Claim~\ref{unions},
    $Y'$ is tame, and by induction, $X' \cap D$ and $X' \setminus D$
    are tame.  By Claim~\ref{unions}, $X'$ is tame.
  \end{claimproof}

  \begin{claim}\label{few-sections}
    Suppose that $\pi : X \to Y$ is a definable surjection with finite
    fibers.  Suppose that $Y$ is tame.  Let $\mathcal{F}$ be a
    definable family of sections of $\pi$.  If $\mathcal{F}$ is
    pseudofinite, then $\mathcal{F}$ is finite.
  \end{claim}
  \begin{claimproof}
    A section is determined by its image.
  \end{claimproof}

  \begin{claim}\label{products}
    Suppose $X$ and $Y$ are tame.  Then so is $X \times Y$.
  \end{claim}
  \begin{claimproof}
    Let $\mathcal{D}$ be a pseudofinite definable family of subsets of
    $X \times Y$.  For each $a \in X$, the set $Y_a := \{a\} \times Y
    \subseteq X \times Y$ is tame, so the collection
    \[ \mathcal{E}_a := \{D \cap Y_a : D \in \mathcal{D}\}\]
    is finite.  Then
    \[ \pi : \coprod_{a \in X} \mathcal{E}_a \to X\] is a definable surjection with finite fibers.  Each element $D \in
    \mathcal{D}$ induces a section of $\pi$, namely, the map
    $\sigma_D$ sending a point $a \in X$ to (the code for) $D \cap
    Y_a$.  This gives a definable injection from $\mathcal{D}$ to
    sections of $\pi$.  By Claim~\ref{few-sections} and the fact that
    $X$ is tame, it follows that $\mathcal{D}$ is finite.
  \end{claimproof}
  It follows that $M^n$ is tame for all $n \ge 1$.  Now if $Y$ is any
  interpretable set, then $Y$ is a set of codes of subsets of $M^n$,
  for some $n$.  By tameness of $M^n$, it follows that if $Y$ is
  pseudofinite, then $Y$ is finite.  This completes the proof of
  Theorem~\ref{main-th}.
\end{proof}

\section{C-minimal expansions of ACVF}
As an example, we apply Theorem~\ref{main-th} to C-minimal expansions
of ACVF.\footnote{See \cite{c-source} for the definition of
  C-minimality.  The theory ACVF is C-minimal by Theorem~4.11 in
  \cite{c-source}.  Certain expansions of ACVF by analytic functions
  are shown to be C-minimal in \cite{lipshitz}.}  Let $T$ be a
C-minimal expansion of ACVF, and $K$ be a sufficiently saturated model
of $T$.  As in the proof of Theorem~\ref{main-th}, work in a setting
with nonstandard counting functions.

\begin{observation}\label{decomp}
  Let $B_1,\ldots,B_n$ be pairwise disjoint balls in $K$.  Then the
  union $\bigcup_{i = 1}^n B_i$ cannot be written as a boolean
  combination of fewer than $n$ balls.
\end{observation}
This follows from uniqueness of the swiss-cheese decomposition, and
the fact that the residue field is infinite.
\begin{lemma}\label{disjoint-case}
  There is no pseudofinite infinite set of pairwise disjoint balls.
\end{lemma}
\begin{proof}
  Let $\mathcal{S}$ be such a set.  By compactness, there must be some
  sequence $\mathcal{S}_1, \mathcal{S}_2, \ldots$ such that each
  $\mathcal{S}_i$ is a finite set of pairwise disjoint balls, the
  $\mathcal{S}_i$ are uniformly interpretable (bounded in complexity),
  and $\lim_{i \to \infty} |\mathcal{S}_i| = \infty$.

  The unions $U_i = \bigcup \mathcal{S}_i \subseteq K$ are
  uniformly definable (bounded in complexity), so there is some
  absolute bound on the number of balls needed to express $U_i$.  But
  Observation~\ref{decomp} says that this number is at least
  $|\mathcal{S}_i|$, a contradiction.
\end{proof}
C-minimality implies that the value group $\Gamma$ is densely
o-minimal.  Therefore $\exists^\infty$ is eliminated in $\Gamma$, and
there are no pseudofinite infinite subsets of $\Gamma$.
\begin{lemma}\label{balls}
  There is no pseudofinite infinite set of balls.
\end{lemma}
\begin{proof}
  Let $\mathcal{S}$ be such a set.  Let $\mathcal{S}_0$ be the set of
  minimal elements of $\mathcal{S}$.  For each $B \in \mathcal{S}_0$,
  let $\mathcal{S}_B$ denote the elements of $\mathcal{S}$ containing
  $B$.  In a pseudofinite poset, every element is greater than or
  equal to a minimal element, so
  \begin{equation*}
    \mathcal{S} = \bigcup_{B \in \mathcal{S}_0} \mathcal{S}_B.
  \end{equation*}
  The set $\mathcal{S}_0$ is pseudofinite, hence finite by
  Lemma~\ref{disjoint-case}.  Therefore, $\mathcal{S}_B$ is infinite
  for some $B$.

  Now $\mathcal{S}_B$ is a chain of balls.  Let $\rho : \mathcal{S}_B
  \to \Gamma$ be the map sending a ball to its radius.  This map is
  nearly injective; the fibers have size at most 2.  The range of
  $\rho$ is pseudofinite, hence finite.  Therefore, the domain
  $\mathcal{S}_B$ is finite, a contradiction.
\end{proof}

Finally, suppose that $\exists^\infty$ is not eliminated on some set
$X_0$ of unary imaginaries.  Then there is a pseudofinite infinite set
$A \subseteq X_0$.  Let $D_a$ be the unary set associated to $a \in
A$.  Note that $a \mapsto D_a$ is injective.

For each $a$, there is a unique minimal set of balls $\mathcal{B}_a$
such that $D_a$ can be written as a boolean combination of
$\mathcal{B}_a$.  The correspondence $a \mapsto \mathcal{B}_a$ is a
definable finite-to-finite correspondence from $A$ to the set $\mathcal{B}$ of balls.
Let $I$ denote the ``image'' of this correspondence:
\begin{equation*}
  I := \bigcup_{a \in A} \mathcal{B}_a.
\end{equation*}
The set $I \subseteq \mathcal{B}$ is pseudofinite, hence finite by
Lemma~\ref{balls}.  The boolean algebra generated by $I$ is finite,
and contains every $D_a$.  By injectivity of $a \mapsto D_a$, the set
$A$ is finite, a contradiction.

By Theorem~\ref{main-th}, we have proven the following:
\begin{proposition}
  $T^{\eq}$ eliminates $\exists^\infty$ when $T$ is a C-minimal
  expansion of ACVF.
\end{proposition}
\begin{acknowledgment}
  The author would like to thank Tom Scanlon, who read an earlier
  version of this paper appearing in the author's disseration.

  This material is based upon work supported by the National Science
  Foundation under Grant No.\ DGE-1106400 and Award No.\ DMS-1803120.
  Any opinions, findings, and conclusions or recommendations expressed
  in this material are those of the author and do not necessarily
  reflect the views of the National Science Foundation.
\end{acknowledgment}

\bibliographystyle{plain}
\bibliography{mybib}{}

\begin{thebibliography}{1}

\bibitem{HHM}
Deirdre Haskell, Ehud Hrushovski, and Dugald Macpherson.
\newblock Definable sets in algebraically closed valued fields: elimination of
  imaginaries.
\newblock {\em J. reine angew. Math.}, 597:175--236, 2006.

\bibitem{unexpected}
Deirdre Haskell, Ehud Hrushovski, and Dugald Macpherson.
\newblock Unexpected imaginaries in valued fields with analytic structure.
\newblock {\em Journal of Symbolic Logic}, 78(2):523--542, June 2013.

\bibitem{lipshitz}
L.~Lipshitz and Z.~Robinson.
\newblock One-dimensional fibers of rigid subanalytic sets.
\newblock {\em Journal of Symbolic Logic}, 63(1):83--88, March 1998.

\bibitem{macintyre}
Angus Macintyre.
\newblock On definable subsets of $p$-adic fields.
\newblock {\em Journal of Symbolic Logic}, 41(3):605--610, September 1976.

\bibitem{c-source}
Dugald Macpherson and Charles Steinhorn.
\newblock On variants of \emph{o}-minimality.
\newblock {\em Annals of Pure and Applied Logic}, 79:165--209, 1996.

\end{thebibliography}

\end{document}